\documentclass[11pt, reqno]{amsart}
\usepackage{amsmath, amsthm, amscd, amsfonts, amssymb, enumerate, mathrsfs, latexsym}
\advance\textheight by \topskip   \numberwithin{equation}{section}

\newtheorem{theorem}{Theorem}[section]
\newtheorem{lemma}[theorem]{Lemma}
\newtheorem{proposition}[theorem]{Proposition}
\newtheorem{corollary}[theorem]{Corollary}

\newtheorem{definition}[theorem]{Definition}

\newtheorem{example}[theorem]{Example}

\numberwithin{equation}{section}

\begin{document}
\title[partitions of multisets]{The Twelvefold Way, the non-intersecting circles problem, and partitions of multisets}
\author{Toufik Mansour}
\address{Department of Mathematics, University of Haifa, 3498838 Haifa, Israel}
\email{tmansour@univ.haifa.ac.il}
\author{Madjid Mirzavaziri}
\address{Department of Mathematics, Ferdowsi university of Mashhad, Mashhad 91775, Iran}
\email{mirzavaziri@um.ac.ir}
\author{Daniel Yaqubi}
\address{Department of Pure Mathematics, Ferdowsi University of Mashhad, Mashhad 91775, Iran}
\email{daniel\_yaqubi@yahoo.es}

 
\begin{abstract}
Let $n$ be a non-negative integer and $\mathbb{A}=\{a_1,\ldots,a_k\}$ be a multiset with $k$ positive integers such that $a_1\leqslant\cdots\leqslant a_k$. In this paper, we give a recursive formula for partitions and distinct partitions of positive integer $n$ with respect to a multiset $\mathbb{A}$. We also consider the extension of the \textit{Twelvefold Way}. By using  this notion, we solve the non-intersecting circles problem which asks to evaluate the number of ways to draw $n$ non-intersecting circles in the plane regardless of their sizes. The latter also enumerates the number of unlabeled rooted tree with $n+1$ vertices.
\end{abstract}
\maketitle

\section{Introduction}
A \textit{partition} of $n$ is a sequence $\lambda_1\geqslant \lambda_2\geqslant \cdots\geqslant \lambda_k$ of positive integers such that $\lambda_1+\lambda_2+\cdots+\lambda_k=n$ (see \cite{And76}). We write $\lambda\vdash n$ to denote that $\lambda$ is a partition of $n$. The non-zero integers $\lambda_k$ in $\lambda$ are called \textit{parts} of $\lambda$. The number of parts of $\lambda$ is the \textit{length} of $\lambda$, denoted by $\ell(\lambda)$, and $\vert\lambda\vert=\sum_{k\geqslant1}\lambda_k$ is the \textit{weight} of $\lambda$. More generally, any weakly decreasing sequence of positive integers is called
a partition. The partition whose parts are
$\lambda_1\geqslant \lambda_2\geqslant \cdots\geqslant \lambda_k$
is usually denoted by $\lambda=(\lambda_1,\lambda_2,\cdots,\lambda_k)$.
Let $P(n)$ denote the set of all partitions of $n$. The size of the set $P(n)$ is denoted by
the \textit{partition function $p(n)$}; that is  $p(n)=\vert P(n)\vert$.
In particular, $p(0)$ consist of a single element,
the unique empty partition of zero, which we denote by $0$.
For example $P(4)$ consists of five elements: $4, 3+1, 2+2, 2+1+1, 1+1+1+1.$ Hence $p(4)=5$.

We let $\mathbb{S}$ be a set of natural numbers and $p(n\lvert \mathbb{S})$
denotes the number of partitions of $n$ into elements
of $\mathbb{S}$ (that is, the parts of the partitions belong to $\mathbb{S}$)
and $p_{\ell}(n\lvert\mathbb{S})$ is the number of partitions of $n$ into exactly $\ell$ parts in $\mathbb{S}$. When $\mathbb{S}=\mathbb{N}$, the set of natural number, we denoted $p_{\ell}(n\lvert\mathbb{N})$ by $p_{\ell}(n)$, that is the number of partitions of $n$
into exactly $\ell$ parts (\textit{or dually, partitions with the largest part equal to $\ell$}).

Recall also that a \textit{multiset} $\mathbb{A}$ with the \textit{multiplicity mapping} $\theta$ is a collection of some not necessarily different objects such that for each
$a\in \mathbb{A}$ the number $\theta(a)$ is the multiplicity of the occurrence of $a$ in $\mathbb{A}$. If $\mathbb{A}$ is a multiset, we denote the set of members of $\mathbb{A}$ by $S(\mathbb{A})$ and we call it \textit{the background set of} $\mathbb{A}$. For a number $a_0$ and a multiset $\mathbb{A}$, the multiset $\{a_0a: a\in \mathbb{A}\}$ is denoted by $a_0\mathbb{A}$. We denote the multiset $\{1,1,\ldots,1\}$ with $\theta(1)=k$ by $I_k$.
We define that $I_0=\emptyset$. Thus, a multiset $\mathbb{A}$ can be
written as $\cup_{i=1}^\ell b_iI_{\theta(b_i)}$, where the background
set $S(\mathbb{A})$ of $\mathbb{A}$ is $\{b_1,\cdots,b_\ell\}$.
For two multisets $\mathbb{A}$ with the multiplicity mapping $\theta_\mathbb{A}$ and $\mathbb{B}$ with the multiplicity mapping $\theta_\mathbb{B}$, we define the multiplicity
mapping $\theta_{\mathbb{A}\setminus \mathbb{B}}$ of $\mathbb{A}\setminus \mathbb{B}$
by $\theta_{\mathbb{A}\setminus \mathbb{B}}(a)=\theta_\mathbb{A}(a)-\theta_\mathbb{B}(a)$ if $\theta_\mathbb{A}(a)\geqslant \theta_\mathbb{B}(a)$
and $\theta_{\mathbb{A}\setminus \mathbb{B}}(a)=0$ if
$\theta_\mathbb{A}(a)<\theta_\mathbb{B}(a)$. Moreover, the multiplicity
mapping $\theta_{\mathbb{A}\cup \mathbb{B}}$ of
$\mathbb{A}\cup \mathbb{B}$ is defined by $\theta_{\mathbb{A}\cup \mathbb{B}}(a)=\theta_\mathbb{A}(a)+\theta_\mathbb{B}(a)$.
In the following $\theta(\mathbb{A})=\sum_{a\in \mathbb{A}}a$, for a multiset $\mathbb{A}$.

For any fixed complex number $|q|\leq1$, any complex number $a$, and any non-negative integer $n$, let
$$(a;q)_n:=\left\{\begin{array}{ll}
\prod_{k=0}^{n-1}(1-aq^k), &  n>0\\
1,                   & n=0.
\end{array}\right.$$
Accordingly, let
\[(a;q)_{\infty}=  \prod_{k=0}^{\infty}(1-aq^k)=\lim_{n\to\infty}(a;q)_n.\]
A \textit{$q$-series} is any series which involves expressions of the form $(a;q)_n$ and $(a;q)_{\infty}$.
The generating function for $p(n)$, discovered by \textit{Euler}, is given by
$\sum_{n=0}^{\infty}p(n)q^n=\frac{1}{(q;q)_{\infty}}$.

The organization of this paper is as follows. In the next section, we consider the number of partitions of $n$ into elements of the multiset $\mathbb{A}$ and the number of partitions of $n$ with distinct parts from the multiset $\mathbb{A}$. As consequence, we present a recursive formula for Wilf's unsolved problem in \cite{Wil2}. In Section 3, we present an extension of the \textit{twelvefold way} which is invented by Stanley, see \cite{Sta}. As consequences, we give a recurrence relation for $B_n$ the number of ways to draw $n$ non-intersecting circles in a plane
regardless to their sizes. In Section 5, we deal with the ordered and unordered factorizations of natural numbers. In Section 6, we present generating functions for our sequences. We end with Section 6, where we establish connections with M\"{o}bius and Euler's totient functions.

\section{Partitions and distinct partitions of positive integer $n$ with respect to a multiset}
Let $n$ be a non-negative integer and $\mathbb{A}=\{a_1,\ldots,a_k\}$ be a multiset with $k$
(not necessarily distinct) positive integers. We denote by $D(n\lvert\mathbb{A})$, the number
of ways to partition $n$ as $a_1x_1+\cdots+a_kx_k$, where $x_i$'s are positive integers
and $x_i\leqslant x_{i+1}$ whenever $a_i=a_{i+1}$. The number of ways to partition $n$
in the form $a_1x_1+\cdots+a_kx_k$, where $x_i$'s are non-negative integers and
$x_i\leqslant x_{i+1}$ whenever $a_i=a_{i+1}$, is also denoted by $D_0(n\lvert\mathbb{A})$.
The numbers $D(n\lvert\mathbb{A})$ and $D_0(n\lvert\mathbb{A})$
are called \textit{the natural partition number} and \textit{the arithmetic partition
number of $n$ with respect to} $\mathbb{A}$.

\begin{lemma}\label{lemma1}
Let $n$ be a non-negative integer and $\mathbb{A}$ be
 a multiset with the multiplicity mapping
$\theta$ and the background set
 $S(\mathbb{A})=\{b_1,\ldots,b_\ell\}$. Then
\[D(n\lvert\mathbb{A})=\sum_{\begin{subarray}{c}
n=b_1y_1+\cdots+b_\ell y_\ell\\ \theta(b_i)\leqslant y_i,~ i=1,
\cdots,\ell\end{subarray}}\prod_{j=1}^\ell p_{\theta(b_j)}(y_j).\]
\end{lemma}
\begin{proof}
Let $A=\{a_1,\ldots,a_k\}$, where $a_i$'s are not necessarily distinct members of $A$.
We can write $n=a_1x_1+\cdots+a_kx_k$ in the form $n=b_1(x_{11}+\cdots+
x_{1\theta(b_1)})+\cdots+b_{\ell}(x_{\ell1}+\cdots+x_{\ell \theta(b_\ell)})$.
Putting $y_i=x_{i1}+\cdots+x_{i\theta(b_i)}$ we have $n=b_1y_1+
\cdots+b_\ell y_\ell$, where $\theta(b_i)\leqslant y_i$ for
$i=1,\cdots,\ell$. Now, the number of ways to partition
$y_i$ in the form $x_{i1}+\cdots+x_{i\theta(b_i)}$
is $p_{\theta(b_i)}(y_i)$, where $1\leqslant x_{i1}\leqslant\cdots\leqslant x_{i\theta(b_i)}$ are positive integers.
\end{proof}
Let $p_{\leqslant m}(n)$ denote the number of partitions of positive integer $n$ into
at most $m$ parts, notice that $p_{\leqslant m}(n)$, is equal to
the number of partitions of positive integer $n$ into parts that are all $\leqslant m$
in view of conjugate partitions. Then
$p_{\leqslant m}(n)=p_0(n)+p_1(n)+\cdots+p_m(n)$.
So we can state the following result.
\begin{lemma}
Let $n$ be a non-negative integer and $\mathbb{A}$ be
 a multiset with the multiplicity mapping
$\theta$ and the background set
 $S(\mathbb{A})=\{b_1,\cdots,b_\ell\}$. Then
\[D_0(n\lvert\mathbb{A})=\sum_{\begin{subarray}{c}
n=b_1y_1+\cdots+b_\ell y_\ell\\ \theta(b_i)\leqslant y_i,~ i=1,
\cdots,\ell\end{subarray}}\prod_{j=1}^\ell p_{\leqslant \theta(b_j)}(y_j).\]
\end{lemma}
Notice, if $n$ is a positive integer and $\mathbb{A}$ is a multiset as $\mathbb{A}=\{1,1,\ldots,1\}$ with multiplicity function $\theta$ that $\theta(1)=\ell$, then
\begin{eqnarray}\label{PLP}
D(n\lvert\mathbb{A})=D(n\lvert\{1,1,\ldots,1\})=p_{\ell}(n).
\end{eqnarray}
That is the number of partitions of positive integer $n$ into exactly $\ell$ parts.
Furthermore, for each multiset $\mathbb{A}$,
$D_0(n\lvert\mathbb{A})=D(n+\theta(\mathbb{A})\lvert\mathbb{A})$,
where $\theta(\mathbb{A})=\sum_{a\in \mathbb{A}}a$.

\begin{proposition}\label{DD}
Let $n$ be a non-negative integer and $\mathbb{A}$ be a multiset with the multiplicity mapping $\theta$. Then for each $a\in \mathbb{A}$,
\begin{eqnarray*}
D(n\lvert\mathbb{A})=\sum_{\begin{subarray}{c}0
\leqslant \ell\leqslant \theta(a)\\ a\theta(a)\leqslant n\end{subarray}}
 D(n-a\theta(a)\lvert\mathbb{A}\setminus aI_\ell),
\end{eqnarray*}
where $D(0\lvert\emptyset)=1$.
\end{proposition}
\begin{proof} Let $\mathbb{A}=\{a_1,\ldots,a_k\}$.
We have $\theta(a)$ occurrence of $a$ in the equation
$n=a_1x_1+\cdots+a_kx_k$.
Let $x_{i+1},\ldots, x_{i+\theta(a)}$
have coefficients $a$ in the equation, $x_{i+1}=\cdots=x_{i+\ell}=1$ and $x_{i+\ell+1}>1$, where $\ell=0,1,\ldots,\theta(a)$. If we subtract $a\theta(a)$ from the both sides of $n=a_1x_1+\cdots+a_kx_k$, then we get
$n-a\theta(a)=a_1x_1+\cdots+a_{i}x_i+a_{i+\ell+1}x_{i+\ell+1}+\cdots+a_kx_k$.
The number of solutions of this equation is
$D(n-a\theta(a)\lvert\mathbb{A}\setminus aI_\ell)$, which completes the proof.
\end{proof}
\begin{example}
We evaluate $D(17\mid\{1,2,2,3\})$ and $D_0(17,\{1,2,2,3\})$.
 Using Proposition~\ref{DD} and Corollary~\ref{OTT} we can write
\begin{align*}
D(17\mid\{1,2,2,3\})&=D(14\mid\{1,2,2,3\})+D(14\mid\{1,2,2\})\\
&=D(11\mid\{1,2,2,3\})+D(11\mid\{1,2,2\})+9\\
&=D(8\mid\{1,2,2,3\})+D(8\mid\{1,2,2\})+6+9\\
&=1+2+6+9=18.
\end{align*}
Moreover,
\begin{align*}
D_0(17,\{1,2,2,3\})&=D(17+8,\{1,2,2,3\})\\
&=D(22,\{1,2,2,3\})+D(22,\{1,2,2\})\\
&=D(19,\{1,2,2,3\})+D(19,\{1,2,2\})+25\\
&=D(16,\{1,2,2,3\})+D(16,\{1,2,2\})+20+25\\
&=D(13,\{1,2,2,3\})+D(13,\{1,2,2\})+12+20+25\\
&=D(10,\{1,2,2,3\})+D(10,\{1,2,2\})+9+12+20+25\\
&=D(7,\{1,2,2,3\})+D(7,\{1,2,2\})+4+9+12+20+25\\
&=0+2+4+9+12+20+25=72.
\end{align*}
\end{example}

\begin{corollary}\label{OTT}
Let $n$ be a positive integer. Then
\begin{align*}
D(n\mid\{1,2\})&=\big\lfloor\frac{n-1}2\big\rfloor,\\
D(n\mid\{1,2,2\})&=\big\lfloor\frac{n-1}{4}\big\rfloor(\big\lfloor\frac{n+1}{2}
\big\rfloor-\big\lfloor\frac{n+3}{4}\big\rfloor),\\
D(n\mid\{1,1,2\})&=\big\lfloor{\frac32\lfloor \frac{n-1}{3}
\rfloor+\frac12}\big\rfloor\big(\big\lfloor\frac{n-1}2\big\rfloor-
\frac12{\big\lfloor\frac32\lfloor \frac{n+2}{3}\rfloor\big\rfloor}+\frac{1+(-1)^n}{2}\big).
\end{align*}
 \end{corollary}
\begin{proof}
Let $n=2k+r$, where $r=1,2$. By Proposition~\ref{DD}, we obtain
\begin{align*}
D(n\lvert\{1,2\})&=D(n-2\lvert\{1,2\})+D(n-2\lvert\{1\})=D(n-2\lvert\{1,2\})+1\\
&=D(n-4\lvert\{1,2\})+D(n-2\lvert\{1\})+1=D(n-4\lvert\{1,2\})+2\\
&=D(n-6\lvert\{1,2\})+3=\cdots=D(n-2k\lvert\{1,2\})+k=0+k=\lfloor\frac{n-1}2\rfloor.
\end{align*}
For the second assertion, let $n=4k+r$, where $r=1,2,3,4$. then
\begin{align*}
D(n\lvert\{1,2,2\})&=D(n-4\lvert\{1,2,2\})+D(n-4\lvert\{1,2\})+D(n-4\lvert\{1\})\\
&=D(n-4\lvert\{1,2,2\})+\lfloor\frac{n-5}2\rfloor+1\\
&=D(n-8\lvert\{1,2,2\})+\lfloor\frac{n-7}2\rfloor+\lfloor\frac{n-3}2\rfloor\\
&=\cdots=D(n-4k\lvert\{1,2,2\})+\sum_{i=1}^k\big \lfloor\frac{n-(4i-1)}2\big\rfloor\\
&=D(r\lvert\{1,2,2\})+\sum_{i=1}^k\big\lfloor\frac{n-(4i-1)}2\big\rfloor=0+\sum_{i=1}^k\big\lfloor\frac{n-(4i-1)}2\big\rfloor\\
&=k\lfloor\frac{n+1}2\rfloor-k(k+1)=\lfloor\frac{n-1}4\rfloor\big(\lfloor\frac{n+1}2\rfloor-\lfloor\frac{n+3}4\rfloor\big),
\end{align*}
as required.
Now, let $n=3k+r$ where $r=1, 2, 3$, then
\begin{eqnarray*}
D(n\mid\{1,1,2\})&=&D(n-2\mid\{1,1,2\})
+D(n-2\mid\{1,2\})+D(n-2\mid\{2\})\\
&=&D(n-2\mid\{1,1,2\})+\lfloor\frac{n-3}2\rfloor
+\frac{1+(-1)^n}{2}\\
&=&D(n-4\mid\{1,1,2\})+\lfloor\frac{n-5}2\rfloor+
\lfloor\frac{n-3}2\rfloor+2\frac{1+(-1)^n}{2}\\
&=&\ldots\\
&=&D(n-4(\lfloor\frac{3k+1}2\rfloor),\{1,1,2\})\\
&&+\sum_{i=1}^{\lfloor\frac{3k+1}2\rfloor}\lfloor\frac{(n-2i)-1}2\rfloor+
\lfloor\frac{3k+1}2\rfloor\frac{1+(-1)^n}{2}\\
&=&0 +\sum_{i=1}^{\lfloor\frac{3k+1}2\rfloor}\lfloor\frac{(n-2i)-1}
2\rfloor+k\frac{1+(-1)^n}{2}\\
&=&\lfloor\frac{3k+1}2\rfloor\lfloor\frac{n-1}2\rfloor
-\frac{\lfloor\frac{3k+1}2\rfloor(\lfloor\frac{3k+1}2\rfloor+1)}{2}
+\lfloor\frac{3k+1}2\rfloor\frac{1+(-1)^n}{2}.
\end{eqnarray*}
It is enough to note that $k=\lfloor\frac{n-1}3\rfloor$.
\end{proof}

Let $Q_m(n)$ be the number of partitions of a positive integer $n$ into exactly
$m$ distinct parts.
It is not difficult to verify by using  by using Ferrers diagrams that
$Q_m(n)=p_{\leqslant m}\left(n-{m+1\choose 2}\right)$,
which means, the number of
partitions of positive integer $n$ into exactly $m$ distinct parts
equals the number of partitions of $n-{m+1\choose 2}$ into at most $m$
parts (\textit{or dually, partitions into parts $\leqslant m$}), see \cite{Com}.
Then, the generating function of $Q_m(n)$ reads as
$$\sum_{n=0}^{\infty}Q_m(n)q^n=\frac{q^{{m+1\choose 2}}}{(q;q)_m}.$$
We let $Q(n)$ is the number of all partitions of $n$ into distinct parts.

Let $n$ be a non-negative integer and $\mathbb{A}=\{a_1,\ldots,a_k\}$ be a multiset of $k$
not necessarily distinct positive integers, where $a_1\leqslant\cdots\leqslant a_k$.
We denote by $\Delta(n\lvert\mathbb{A})$ the number of partitions of $n$ as
the form $a_1x_1+\cdots+a_kx_k$, where $x_i$'s are distinct positive integers and
$x_i< x_{i+1}$ whenever $a_i=a_{i+1}$. The number of
 partitions of $n$ of the form $a_1x_1+\cdots+a_kx_k$, where $x_i$'s are distinct
non-negative integers and $x_i< x_{i+1}$ whenever $a_i=a_{i+1}$, is also denoted by
$\Delta_0(n\lvert\mathbb{A})$. The numbers $\Delta(n\lvert\mathbb{A})$ and
$\Delta_0(n\lvert\mathbb{A})$ are called the \textit{natural distinct partition number}
and the \textit{arithmetic distinct partition number of $n$ with respect to} $\mathbb{A}$.

\begin{lemma}
Let $n$ be a non-negative integer and $\mathbb{A}$ be
 a multiset with the multiplicity mapping
$\theta$ and the background set
 $S(\mathbb{A})=\{b_1,\ldots,b_\ell\}$. Then
\[\Delta(n\lvert\mathbb{A})=\sum_{\begin{subarray}{c}
n=b_1y_1+\cdots+b_\ell y_\ell\\ \theta(b_i)\leqslant y_i,~ i=1,
\cdots,\ell\end{subarray}}\prod_{j=1}^\ell Q_{\theta(b_j)}(y_j).\]
\end{lemma}
\begin{proof}
Proof as similar  to Lemma \ref{lemma1}. Let $\mathbb{A}=\{a_1,\ldots,a_k\}$, where $a_i$'s are $k$ not necessarily distinct members of $A$.
We can write $n=a_1x_1+\ldots+a_kx_k$ as the form $n=b_1(x_{11}+\ldots+
x_{1\theta(b_1)})+\ldots+b_{\ell}(x_{\ell1}+\ldots+x_{\ell \theta(b_\ell)})$.
Putting $y_i=x_{i1}+\ldots+x_{i\theta(b_i)}$ we have $n=b_1y_1+
\ldots+b_\ell y_\ell$, where $\theta(b_i)\leqslant y_i$ for
$i=1,\ldots,\ell$. Now the number of ways to partition
$y_i$ into $x_{i1}+\ldots+x_{i\theta(b_i)}$
with $1\leqslant x_{i1}\leqslant\ldots\leqslant x_{i\theta(b_i)}$ is $Q_{\theta(b_i)}(y_i)$.
\end{proof}
Let $Q_{\leqslant m}(n)$ denote the number of partitions of positive
integer $n$ into at most $m$ distinct parts. Then
$Q_{\leqslant m}(n)=Q_1(n)+Q_2(n)+\cdots+Q_m(n)$, which leads to the following corollary.
\begin{corollary}
Let $n$ be a non-negative integer and $\mathbb{A}$ be
 a multiset with the multiplicity mapping
$\theta$ and the background set
 $S(\mathbb{A})=\{b_1,\ldots,b_\ell\}$. Then
\[\Delta_0(n\lvert\mathbb{A})=\sum_{\begin{subarray}{c}
n=b_1y_1+\cdots+b_\ell y_\ell\\ \theta(b_i)\leqslant y_i,~ i=1,
\cdots,\ell\end{subarray}}\prod_{j=1}^\ell Q_{\leqslant \theta(b_j)}(y_j).\]
\end{corollary}
\begin{corollary}\label{Dis}
 Let $n$ be a non-negative integer and $\mathbb{A}$ be
 a multiset with the multiplicity mapping
$\theta$ and the background set
 $S(\mathbb{A})=\{b_1,\ldots,b_\ell\}$. Then
 \begin{eqnarray*}
  \Delta \left( n\lvert\mathbb{A}\right)=D\left( n+\theta(\mathbb{A})
  -\sum_{i=1}^{\ell}b_i{\theta(b_i)+1\choose 2}\lvert\mathbb{A}\right).
 \end{eqnarray*}
\end{corollary}
\begin{proof}
 Let $n$ be a non-negative integer and $\mathbb{A}=\{a_1,\ldots,a_k\}$ be a multiset with $k$
not necessarily distinct positive integers, where $a_1\leqslant\cdots\leqslant a_k$.
$\Delta(n\lvert\mathbb{A})$ the number of partitions of $n$ as
the form $n=a_1x_1+\cdots+a_kx_k$, where $x_i$'s are distinct positive integers and
$x_i< x_{i+1}$ whenever $a_i=a_{i+1}$. We can write
\[n=b_1(x_{11}+\cdots+x_{1\theta(b_1)})+\cdots+b_{\ell}(x_{\ell1}+\cdots+x_{\ell \theta(b_\ell)}).\]
Putting $y_i=x_{i1}+\cdots+x_{i\theta(b_i)}$. The number of partitions
of $y_i$ into exactly $\theta(b_i)$ distinct parts equal $Q_{\theta(b_i)}(y_i)$ for
$i=1,\ldots,\ell$. By Corollary \eqref{Dis}, we get
$$Q_{\theta(b_i)}(y_i)=p_{\leqslant \theta(b_i)}\left( y_i+{\theta(b_i)+1\choose 2}\right).$$
Then, we can write
\begin{eqnarray*}
 n&=&b_1y_1+b_2y_2+\cdots+b_\ell y_\ell\\
 &=&b_1\left( y_1-{\theta(b_1)+1\choose 2}\big)+\cdots+b_\ell\big(y_\ell-{\theta(b_\ell)+1\choose 2}\right)\cr
 &=&b_1y_1+b_2y_2+\cdots+b_\ell y_\ell-\sum_{i=1}^{\ell}b_i{\theta(b_i)+1\choose 2}.
\end{eqnarray*}
Then, we can conclude
\[\Delta(n\lvert\mathbb{A})=D_0\left((n
  -\sum_{i=1}^{\ell}b_i{\theta(b_i)+1\choose 2}\lvert\mathbb{A}\right)
  =D\left( n+\theta(\mathbb{A})
  -\sum_{i=1}^{\ell}b_i{\theta(b_i)+1\choose 2}\lvert\mathbb{A}\right),\]
as claimed.
\end{proof}

It is easy to see that if $n$ is a positive integer and $\mathbb{A}$ be the multiset
$\{1,1,\ldots,1\}$, with multiplicity function $\theta(1)=\ell$, then
$\Delta(n\lvert\mathbb{A})=\Delta(n\lvert\{1,1,\ldots,1\})=Q_{\ell}(n)$.
Furthermore, for each multiset $\mathbb{A}$,
$\Delta_0(n\lvert\mathbb{A})=\Delta(n+\theta(\mathbb{A})\lvert\mathbb{A})$,
where, $\theta(\mathbb{A})=\sum_{a\in \mathbb{A}}a$.

 \textit{Herbert Wilf} posed some unsolved problems in \cite{ Wil2}.
 Wilf's Sixth Unsolved Problem regards
 \textit{``the set of partitions of positive integer $n$ for which
the (nonzero) multiplicities of its parts are all different''}.
We refer to these as \textit{Wilf partitions} and  $T(n)$ for the set of
Wilf partitions. For example there exist $4$ Wilf partitions of $n=4$ :
\[4=(1)4;\quad 2+2=(2)2 ;\quad 2+1+1=(1)2+(2)1;\quad 1+1+1+1=(4)1.\]
Then, $\vert T(4)\vert=4$.
Let $\mathbb{A}=\{a_1,a_2,\ldots,a_k\}$ be a set of non-negative integers. We denote $T(n\lvert\mathbb{A})$ for the
number of Wilf partitions of positive integers $n$ as the form
$a_1x_1+a_2x_2+\cdots+a_kx_k$, where $x_i$'s are positive distinct integers.
Furthermore, if we put $\mathbb{A}=\mathbb{N}$, the set of natural numbers,
then $T(n\lvert\mathbb{A})=|T(n)|$.
\begin{proposition}\label{delta}
Let $n$ be a non-negative integer and $\mathbb{A}=\{a_1,\ldots,a_k\}$
be a multiset with the background set
$S(\mathbb{A})=\{b_1,\ldots,b_\ell\}$. Then
$$\Delta(n\lvert\mathbb{A})=\Delta(n-\theta(\mathbb{A})\lvert\mathbb{A})
+\sum_{i=1}^\ell\Delta(n-\theta(\mathbb{A})\lvert\mathbb{A}\setminus\{b_i\}).$$
Moreover,
$\Delta(n\lvert\mathbb{A})=0$ when $n<\sum_{i=1}^k(k+1-i)a_i$.
\end{proposition}
\begin{proof}
At most one of $x_i$'s can be $1$. If there is no $x_i$ with $x_i=1$ then we can
write $n-\theta(A)=a_1(x_1-1)+\cdots+a_k(x_k-1)$ and there are
$\Delta(n-\theta(A),A)$ solutions for this equation under the required conditions.
Moreover, if $x_j=1$ for some $j$, then other $x_i$'s are greater that $1$ and thus
we can write
\[n-\theta(A)=a_1(x_1-1)+\cdots+a_{j-1}(x_{j-1}-1)+a_{j+1}(x_{j+1}-1)+\cdots+a_k(x_k-1).\]
There are $\Delta(n-\theta(A)\lvert\mathbb{A}\setminus\{b_i\})$ solutions for the latter equation,
where $b_i=a_j$.  The other parts are obvious.
\end{proof}
\begin{corollary}
Let $n$ be a non-negative integer and $\mathbb{A}=\{a_1,\ldots,a_k\}$
be a set of non-negative integers. Then $T(n\lvert \mathbb{A})$ given by
$T(n\lvert\mathbb{A})=\sum_{i=0}^k T(n-\theta(\mathbb{A})\lvert\mathbb{A}\setminus\{a_i\})$
with $b_0=\emptyset$.
\end{corollary}

\begin{example}
We evaluate $\Delta(18\lvert\{1,2,2,3\})$. By Proposition \ref{delta}, we have
\begin{eqnarray*}
\Delta(18\lvert\{1,2,2,3\})&=&\Delta(10\lvert\{1,2,2,3\})+\Delta(10\lvert\{2,2,3\})
+\Delta(10\lvert\{1,2,3\})+\Delta(10\lvert\{1,2,2\})\\
&=&0+0+\Delta(4\lvert\{1,2,3\})+\Delta(4\lvert\{2,3\})
+\Delta(4\lvert\{1,3\})+\Delta(4\lvert\{1,2\})\\&&
+\Delta(5\lvert\{1,2,2\})+\Delta(5\lvert\{2,2\})+\Delta(5\lvert\{1,2\})\\
&=&0+0+0+0+0+1+0+0+2=3.
\end{eqnarray*}
The 3 solutions are
\begin{eqnarray*}
18&=&1\times\mathbf{3}+2\times\mathbf{2}+2\times\mathbf{4}+3\times\mathbf{1}=1\times\mathbf{5}+2\times\mathbf{2}+2\times\mathbf{3}+3\times\mathbf{1}\\
&=&1\times\mathbf{4}+2\times\mathbf{1}+2\times\mathbf{3}+3\times\mathbf{2}.
\end{eqnarray*}
\end{example}
\begin{corollary}\label{coro}
Let $n$ be a positive integer. Then
$$\Delta(n\mid\{1,1\})=\lfloor\frac{n-1}2\rfloor \mbox{ and }
\Delta(n\mid\{1,2\})=\lfloor\frac{n-1}3\rfloor+\lfloor\frac{n-1}6\rfloor.$$
\end{corollary}
\begin{proof}
Let $n=2k+r$, where $r=1,2$. Using Proposition \ref{delta} we can write
\begin{align*}
\Delta(n\mid\{1,1\})&=\Delta(n-2\mid\{1,1\}) + \Delta(n-2\mid\{1\})
\\&=\Delta(n-2\mid\{1\}) + 1
\\&=\Delta(n-4\mid\{1,1\}) + \Delta(n-4\mid\{1\}) +1
\\&=\ldots
\\&=\Delta(n-2k\mid\{1,1\}) + k=0 + k=\lfloor\frac{n-1}2\rfloor.
\end{align*}
Now let $n=3k+r$, where $r=1,2,3$. Thus
\begin{align*}
\Delta(n\mid\{1,2\})&=\Delta(n-3\mid\{1,2\}) +
\Delta(n-3\mid\{1\}) +\Delta(n-3\mid\{2\})
\\&=\Delta(n-3\mid\{1,2\}) + \Delta(n-3\mid\{2\}) +1
\\&=\Delta(n-6\mid\{1,2\}) + \Delta(n-6\mid\{1\}) + \Delta(n-6\mid\{2\}) +
\Delta(n-3\mid\{2\})+1
\\&=\Delta(n-6\mid\{1,2\}) + \Delta(n-6\mid\{2\}) + \Delta(n-3\mid\{2\})+ 2
\\&=\ldots
\\&=\Delta(n-3k\mid\{1,2\}) + \sum_{i=1}^{\lfloor\frac{n-1}{3}\rfloor}\Delta(n-3i\mid\{2\})
+ k
\\&=0 + \sum_{i=1}^{\lfloor\frac{n-1}{3}\rfloor}\Delta(n-3i\mid\{2\}) +
\lfloor\frac{n-1}{3}\rfloor.
\end{align*}
If $n=3k$ then $k-i$ is even and so
\[\sum_{i=1}^{\lfloor\frac{n-1}{3}\rfloor}\Delta(n-3i\mid\{2\})
+\lfloor\frac{n-1}{3}\rfloor=\lfloor\frac{n-1}6\rfloor+\lfloor\frac{n-1}3\rfloor.\]
Similarly, we have the result for the cases $n=3k+1$ and $n=3k+2$.
\end{proof}
\section{The twelvefold way}
The \textit{Twelvefold Way} gives the number of mappings $f$ from the set $N$
of $n$ objects to set $K$ of the $k$ objects
(putting balls from the set $N$ into boxes in the set $K$). \textit{Richard Stanley} invented
the twelvefold way \cite{Sta}. Consider $n$ (un)labeled balls and $k$  (un)labeled cells.
There are four cases: $\mathbf{U}\rightarrow \mathbf{L}, \mathbf{L} \rightarrow\mathbf{ U},
\mathbf{L} \rightarrow\mathbf{L}, \mathbf{U} \rightarrow\mathbf{U}$,
for arrangements of $\mathbf{L}$abeled or $\mathbf{U}$nLabeled balls
$\hspace{0.1cm}\xrightarrow{in}\hspace{0.2cm}\mathbf{L}$abeled or $\mathbf{U}$nLabeled
boxes. Here Labeled means distinguishable and unLabeled means indistinguishable.
If we want to partition these balls into these
cells we are faced with the following  twelve problems (see Table \ref{Table:1}).
\begin{table}[htp]
 \centering
\begin{tabular}{ |c||c|c|c|c| }\hline
elements of $N$& elements of $K$& $f$ unrestricted&$f$ one-to-one&$f$ onto\\
\hline\hline
 $\mathbf{L}$&$\mathbf{L}$ &$k^n$&$(k-n+1)_n$&$k!{n\brace k}$\\
 \hline
$\mathbf{U}$& $\mathbf{L}$&${n+k-1\choose n}$ &${k\choose n}$&${n-1\choose n-k}$\\
\hline
$\mathbf{L}$ &$\mathbf{U}$& $\sum_{i=1}^k{n\brace i}$&$\delta_{k\leqslant n}$&${n\brace k}$\\
\hline
 $\mathbf{U}$&$\mathbf{U}$& $\sum_{i=1}^i p_i(n)$&$\delta_{k\leqslant n}$&$p_k(n)$\\
\hline
\end{tabular}
\caption{The Twelvefold Way}
\label{Table:1}
\end{table}
In Table \ref{Table:1}, $(k)_n:=k(k-1)\cdots(k-n+1)$ is the \textit{Pochhammer's symbol} or
\textit{falling factorial}, for $k,n\in \mathbb{N}$, ${n\brace k}$
denotes the {\em Stirling number of the second kind}, that is the number of partitions of the set $\{1,2,\ldots,n\}$ into exactly $k$ non-empty subsets, which is equal to $\sum_{i=1}^k(-1)^i{k\choose i}(k-i)^n$, the number ${n\brace k}$ satisfies the recursive relation
${n\brace k}={n-1\brace k-1}+k{n-1\brace k}$ and $\delta_{k\leqslant n}:=
\left\{
   \begin{array}{ll}
                1                   & {\rm when~} n\leqslant k,\\
                0                   & {\rm when~} n>k.
         \end{array}
         \right.$
Now, we consider a new problem as an extension and unification of the above problems.
Consider $b_1+b_2+\cdots+b_n$ balls with $b_1$ balls Labeled $1$, $b_2$ balls Labeled $2$,
and so on, $c_1+c_2+\cdots+c_k$ cells with $c_1$
cells Labeled $1$, $c_2$ cells Labeled $2$, and so on. We denote the situation of these
balls and cells by the two multisets $\mathcal{B}=\{b_1,b_2,\ldots,b_n\}$
of balls and $\mathcal{C}=\{c_1,\ldots,c_k\}$ of cells.
Let the number of mappings $\mathcal{F}$ from the multiset $\mathcal{B}$ of balls
to the multiset $\mathcal{C}$ of cells, be called \textit{The Mixed Twelvefold Way}
\textit{(or dually, the number of ways to partition the multiset $\mathcal{B}$ of balls into
the multiset $\mathcal{C}$ of cells)}.
We denote the number of \textit{unrestricted} mappings of
$\mathcal{F}$ by $\Gamma_0(\mathcal{B}\lvert\mathcal{C})$.
Also, we denote the number of \textit{onto} mappings of $\mathcal{F}$, that is,
\textit{ the number of ways
to partition the multiset $\mathcal{B}$ of balls into the multiset
 $\mathcal{C}$ of cells, such that the cells are nonempty}
by $\Gamma(\mathcal{B}\lvert\mathcal{C})$.

\begin{theorem}\label{pi0}
Let $\mathcal{B}=\{b_1,\ldots,b_n\}$ and $\mathcal{C}=\{c_1,\ldots,c_k\}$
be two multisets whose members are positive integers. The number of unrestricted
mappings $\mathcal{F}$ from the
multiset $\mathcal{B}$ to $\mathcal{C}$ is given by
\begin{eqnarray*}
\Gamma_0(\mathcal{B}\lvert\mathcal{C})=\sum_{\substack {b_1
=n_1+\cdots+n_k\\ 0\leqslant n_i\leqslant b_1}}
\sum_{\begin{subarray}{c}(C_1,\cdots,C_k)\\ \theta(C_i)\leqslant c_i\end{subarray}}
\big(\prod_{j=1}^k\Delta(n_j\lvert C_j)\big)
\Gamma_0\big(\mathcal{B}\setminus\{b_1\}\lvert
(\bigcup_{i=1}^kC_i)\cup(\bigcup_{\substack{i=1\\
\theta(C_i)<c_i}}^k\{c_i-\theta(C_i)\}\big),
\end{eqnarray*}
where $\Gamma_0(\emptyset,A)=1$ for each multiset $A$ of nonnegative integers.
\end{theorem}
\begin{proof}
At first, we distribute the $b_1$ balls Labeled $1$ into cells.
Let $n_i$ be the number of balls in cells Labeled $i$ for $i=1,\ldots,k$.
Thus we can write $b_1=n_1+\cdots+n_k$. When we put $n_i$ balls in cell
Labeled $i$, the $c_i$ cells Labeled $i$ partitioned into different types.
Suppose that we have $\ell_{ij}$ cells Labeled $i$ with $x_{ij}$ balls
Labeled $1$. Whence $c_i=\ell_{i1}x_{i1}+\cdots+\ell_{it}x_{it}+r_i$,
where $r_i$ is the number of cells Labeled $i$ which are still empty.
Let $C_{i}=\{\ell_{i1},\ldots,\ell_{it}\}$. Thus $\theta(C_i)\leqslant c_i$
and there are $\Delta(n_i\lvert C_i)$ situations in which the types of the
$c_i$ cells Labeled $i$ change into $\ell_{i1}$ cells of the first type,
say Labeled $\ell_{i1}$, $\ldots$, $\ell_{it}$ cells of the $t$-th type, say
Labeled $it$, and $r_i$ empty cells the $t+1$-st type, say Labeled $i(t+1)$.
We can therefore say that after distributing the $b_1$ balls Labeled
$1$ into cells we have the multiset $\mathcal{B}\setminus\{b_1\}$ of
balls and the multiset
$$(\bigcup_{i=1}^kC_i)\cup(\bigcup_{\substack{i=1\\
\theta(C_i)<c_i}}^k\{c_i-\theta(C_i)\}\big),$$
of cells.
The number of ways putting of these balls into these cells is
$$\Gamma_0\big(\mathcal{B}\setminus\{b_1\}\lvert
(\bigcup_{i=1}^kC_i)\cup(\bigcup_{\substack{i=1\\
\theta(C_i)<c_i}}^k\{c_i-\theta(C_i)\}\big),$$
which completes the proof.
\end{proof}
\begin{theorem}\label{pi}
Let $\mathcal{B}=\{b_1,\ldots,b_n\}$ and
$\mathcal{C}=\{c_1,\ldots,c_k\}$
be two mulitsets whose members are positive integers.
The number of onto
mappings of $\mathcal{F}$ from the
multiset $\mathcal{B}$ to $\mathcal{C}$ is given by
\begin{align*}
\Gamma(\mathcal{B}\lvert\mathcal{C})=
\sum_{\ell=0}^k\sum_{1\leqslant i_1<\cdots<i_\ell\leqslant k}(-1)^\ell
\Gamma_0\big(\mathcal{B}\lvert\bigcup_{i=1}^\ell
\big((\mathcal{C}\setminus\{c_{i_j}\})\cup(\{c_{i_j}-1\})\big)\big).
\end{align*}
\end{theorem}
\begin{proof}
Let ${\mathcal E}_i$ be the set of situations in which some of
cells Labeled $i$ is empty. Then the number of the elements of
${\mathcal E}_{i_1}\cap\cdots\cap{\mathcal E}_{i_\ell}$ is $\Gamma_0\big(\mathcal{B}\lvert
\bigcup_{i=1}^\ell\big((\mathcal{C}\setminus\{c_{i_j}\})\cup(\{c_{i_j}-1\})\big)\big)$.
Now the inclusion exclusion principle implies the result.
\end{proof}
Let $n$ and $k$ be positive integers. Consider $\mathcal{B}=\{1,2,\ldots,n\}$, the set of $n$ unLabeled balls and
$\mathcal{C}=\{1,2,\ldots,k\}$ the set of $k$ unLabeled cells, also,
$\mathcal{I}_k=\{1,1,\ldots,1\}$ be a multiset with multiplicity mapping $m$,
such that $\theta(1)=k$. Then, we conclude the following result about,
the number of unrestricted or onto mappings of $\mathcal{F}$, from the set
$\mathcal{B}$ or $\mathcal{I}_k$
to the set $\mathcal{C}$ or $\mathcal{I}_k$.
Then
\begin{itemize}
\item [i)] {$\Gamma(\mathcal{B}\lvert\mathcal{C})=p_k(n)$ and
$\Gamma_0(\mathcal{B}\lvert\mathcal{C})=p_k(n+k)$.}
\item [ii)] {$\Gamma(\mathcal{B}\lvert\mathcal{I}_k)={n-1\choose k-1}$ and
$\Gamma_0(\mathcal{B}\lvert\mathcal{I}_k)={n+k-1\choose k-1}$.}
\item [iii)] {$\Gamma(\mathcal{I}_n\lvert\mathcal{C})={n\brace k}$ and
$\Gamma_0(\mathcal{I}_n\lvert\mathcal{C})=\sum_{i=1}^k{n\brace i}$.}
\item [iv)] {$\Gamma(\mathcal{I}_n\lvert\mathcal{I}_k)=k!{n\brace k}$ and
$\Gamma_0(\mathcal{I}_n\lvert\mathcal{I}_k)=k^n$.}
\end{itemize}
\begin{corollary}
Let $n$ and $k$ be positive integers. Then
$p_k(n)=\sum_{\theta(\mathcal{C})=k}\Delta(n\lvert\mathcal{C})$,
where the summation is taken over all multisets $\mathcal{C}$
whose members are positive integers.
\end{corollary}
\begin{proof}
Using Theorems \ref{pi0} and \ref{pi}, we can write
\begin{eqnarray*}
p_k(n)&=&\Gamma(\{1,2,\ldots,n\}\lvert\{1,2,\ldots,k\})\\
&=&\Gamma_0(\{1,2,\ldots,n\}\lvert\{1,2,\ldots,k\})
-\Gamma_0(\{1,2,\ldots,n\}\lvert\{1,2,\ldots,k-1\})\\
&=&\sum_{\theta(\mathcal{C})\leqslant k}\Delta(n\lvert\mathcal{C})
-\sum_{\theta(\mathcal{C})\leqslant k-1}\Delta(n\lvert\mathcal{C})=\sum_{\theta(\mathcal{C})=k}\Delta(n\lvert\mathcal{C}),
\end{eqnarray*}
as claimed.
\end{proof}
\section{the non-intersecting circles problem}
To solve the non-intersecting circles problem, let us assume the following notations.
Let $n$ be a positive integer. We denote the set of all multisets $\mathbb{A}=\{a_1,\ldots,a_k\}$
such that there are distinct positive integers $x_1,\ldots,x_k$ with $n=a_1x_1+\cdots+a_kx_k$,
where $x_i<x_{i+1}$ whenever $a_i=a_{i+1}$, by ${\mathcal A}_{n,k}$. Recall that for an
$\mathbb{A}\in{\mathcal A}_{n,k}$ there are $\Delta(n\lvert\mathbb{A})$ solutions $(x_1,\ldots,x_k)$
satisfying the above condition. We denote the set of these $(x_1,\ldots,x_k)$ by ${\mathcal X}_\mathbb{A}$.

Note that the number of $(n_1,\ldots,n_r)$ with
$1\leqslant n_1\leqslant \cdots\leqslant n_r\leqslant s$ is given by
\begin{eqnarray}\label{parts}
\sum_{i=1}^k{r-1\choose i-1}{s\choose i}=\sum_{i=1}^k{r-1\choose r-i}{s\choose i}={r+s-1\choose r}.
\end{eqnarray}

The non-intersecting circles problem asks to evaluate the number of ways to draw $n$ non-intersecting
circles in a plane regardless to their sizes. For example, if we use the symbol $(~)$
for a circle then there are four such ways for 3,
circles
$(~)(~)(~),((~)(~)),((~))(~),(((~)))$
and nine ways for 4 circles,
$$(~)(~)(~)(~),
((~)(~)(~)),
((~)(~))(~),
(((~)(~))),((~))(~)(~),
(((~))(~)),
(((~)))(~),
((((~)))),((~))((~)).$$
If we denote this number by $B_n$ then we can see that $B_0=B_1=1, B_2=2, B_3=4, B_4=9$, $B_5=20$ and so on.
\begin{theorem}
Let $B_n$ be the number of ways to draw $n$ non-intersecting circles in a plane
regardless to their sizes. Then
\[B_n=\sum_{k=1}^{\lfloor\sqrt{2n}\rfloor}\sum_{\mathbb{A}=
\{a_1,\cdots,a_k\}\in{\mathcal A}_{n,k}}
\sum_{(x_1,\cdots,x_k)\in {\mathcal X}_{A}}\prod_{i=1}^k {B_{x_i-1}+a_i-1\choose a_i}.\]
\end{theorem}
\begin{proof}
Given $n$, let us draw our circles in $\ell$ parts with $y_i$ circles in $i$-th part. We can assume that $y_1\leqslant\cdots\leqslant y_\ell$. Thus $n=y_1+\cdots+y_\ell$. We can rewrite it as the form $n=a_1x_1+\cdots+a_kx_k$ such that $x_i<x_{i+1}$ whenever $a_i=a_{i+1}$. This shows that we have $a_i$ parts with $x_i$ circles of the form $(x_i-1)$ where $(~)$ denotes a circle containing $x_i-1$ circles. We can form the $a_i$ parts of the form $(x_i-1)$ in ${B_{x_i-1}+a_i-1\choose a_i}$ ways. The latter is true since we may put $r=a_i$ and $s=B_{x_i-1}$ in ~\ref{parts}. Note that a single form $(x_i-1)$ can be drawn in $B_{x_i-1}$ ways.

Now notice the fact that the maximum of $k$ occurs when $a_1=\cdots=a_k=1$. Since we have $1\leqslant x_1<\cdots<x_k$ in this case, we can therefore deduce that $\frac{k(k+1)}2\leqslant n$. Thus the maximum value of $k$ is $\lfloor\sqrt{2n}\rfloor$.
\end{proof}
\begin{example}
For $n=6$ we have
\begin{eqnarray*}
{\mathcal A}_{6,1}&=&\{\{\mathbf{1}\},\{\mathbf{2}\},\{\mathbf{3}\},\{\mathbf{6}\}\}\\
{\mathcal A}_{6,2}&=&\{\{\mathbf{1},\mathbf{1}\},\{\mathbf{1},\mathbf{2}\},\{\mathbf{1},\mathbf{3}\},\{\mathbf{1},\mathbf{4}\},\{\mathbf{2},\mathbf{2}\}\}\\
{\mathcal A}_{6,3}&=&\{\{\mathbf{1},\mathbf{1},\mathbf{1}\}\}
\end{eqnarray*}
We can therefore write
\begin{eqnarray*}
6&=&\mathbf{1}\times 6=\mathbf{2}\times 3=\mathbf{3}\times 2=\mathbf{6}\times 1=\mathbf{1}\times 1+\mathbf{1}\times 5=\mathbf{1}\times 2+\mathbf{1}\times 4\\
&=&\mathbf{1}\times 4+\mathbf{2}\times 1=\mathbf{1}\times 3+\mathbf{3}\times 2=\mathbf{1}\times 2+\mathbf{4}\times 1=\mathbf{2}\times 1+\mathbf{2}\times 2\\
&=&\mathbf{1}\times 1+\mathbf{1}\times 2+\mathbf{1}\times 3.
\end{eqnarray*}
Thus
\begin{eqnarray*}
B_6&=&{B_5\choose 1}+{B_2+1\choose 2}+{B_1+2\choose 3}+{B_0+5\choose 6}\\&&+{B_0\choose 1}{B_4\choose 1}+{B_1\choose 1}{B_3\choose 1}
+{B_3\choose 1}{B_0+1\choose 2}+{B_2\choose 1}{B_1+2\choose 3}\\&&+{B_1\choose 1}{B_0+3\choose 4}+{B_0+1\choose 2}{B_1+1\choose 2}+{B_0\choose 1}{B_1\choose 1}{B_2\choose 1}\\
&=&20+3+1+1+9+4+4+2+1+1+2=48.
\end{eqnarray*}
So, the number of ways to draw $6$ non-intersecting circles in a plane
regardless to their sizes is equal $48$.
\end{example}
A rooted tree may be defined as a free tree  in which some vertex has been distinguished as the \textit{root}. We can see some values of rooted tree for positive integer $n$ in \cite{In1}.
\begin{corollary}
Let $n$ be a positive integer. Then $B_n$ is the number of unlabeled rooted tree with $n+1$ vertices.
\end{corollary}
\begin{proof}
There is a one to one correspondence between $n$ non-intersecting circles and an unlabeled rooted tree with $n+1$ vertices. It is enough to draw a circle for each non-root vertex and put a circle inside another one if the second one is the parent of the first one.
\end{proof}
\section{Ordered and unordered factorizations of natural numbers}
An \textit{ordered factorization} of a positive integer $n$ is a representation of $n$ as an ordered
product of integers, each factor greater than $1$.
For positive integer $\ell,k\geqslant1$ we denoted the number of ordered factorization of
positive integer $n$ into exactly $k$ factors, such that each factors
$\geqslant\ell$ by $\mathcal{H}(n;k,\ell)$.
 We use $\mathcal{H}(n)$ to represent the number of all
 ordered factorization of the positive integer $n$ \textit{(in analogy with compositions for sum)}.
 For example, $\mathcal{H}(12)=8$, since we have the factorizations $12,2\times6,6\times2,3\times4,
  4\times3,2\times2\times3,2\times3\times2$ and $3\times2\times2$. By the definition, $\mathcal{H}(1)=0$,
  but some situations it is useful to set $\mathcal{H}(1)=1$ or $\mathcal{H}(1)=\frac{1}{2}$, \cite{Hil}.\\
  Every integer $n>1$ has a canonical factorization into distinct prime numbers $p_1,p_2,\ldots,p_r$, namely
  \begin{eqnarray}\label{Pr}
  n=p^{\alpha_1}p^{\alpha}_2\ldots p^{\alpha_r};\quad\quad 1<p_1<p_2<\ldots<p_r.
  \end{eqnarray}
  Many problems involving factorisatio numerorum depend only on the set of exponents in
  \ref{Pr}, $\{\alpha_1,\alpha_2,\ldots,\alpha_r\}$.
  \textit{MacMahon} \cite{Mac93} developed the theory of compositions of \textit{multipartite numbers}
  from this perspective, and indeed considered these problems throughout his career
  \cite{Mac27}, but \textit{Andrews} suggests the more modern
  terminology \textit{vector compositions} \cite{And76}, p.$57$.
  A general formula for $\mathcal{H}(n,k,2)$ of ordered factorizations of
  positive integer $n$ such that each factor larger than $2$ given by MacMahan in \cite{Mac93}.
  Now, we give another proof for $\mathcal{H}(n,k,2)$ and $\mathcal{H}(n,k,1)$ with above results.

\begin{theorem}
Let $n=p_1^{\alpha_1}\ldots p_r^{\alpha_r}$ be a positive integer. Then,
the number of ordered factorizations of $n$ into $k$ factors such that
each factor $\geqslant1$ and $\alpha_1+\ldots+\alpha_n\geqslant k\geqslant 1$, is given by
\begin{eqnarray}
 \mathcal{H}(n,k,1)=\sum_{i=1}^{\alpha_1+\ldots+\alpha_r}\Gamma_0(\{\alpha_1,\ldots,\alpha_r\},I_i)
=\sum_{i=1}^{\alpha_1+\ldots+\alpha_r}\prod_{j=1}^n{\alpha_j+i-1\choose i-1}.
\end{eqnarray}
Also, the number of unordered factorizations of $n$ into $k$ factors such that
each factor $\geqslant 2$ is given by\\
\begin{eqnarray}
\mathcal{H}(n,k,2)&=&\sum_{i=1}^{\alpha_1+\ldots+\alpha_r}\Gamma(\{\alpha_1,\ldots,\alpha_r\},I_i)\cr
&=&\sum_{i=1}^{\alpha_1+\ldots+\alpha_r}
\sum_{\ell=0}^i(-1)^\ell{i\choose \ell}\prod_{j=1}^n{\alpha_j+i-\ell-1\choose i-\ell-1}.
\end{eqnarray}
\end{theorem}
\begin{proof}
It is sufficiently using of theorems \ref{pi} and theorem \ref{pi0}.
Suppose that for $1\leqslant j\leqslant n$ we have $\alpha_j$ balls
Labelled $p_j$ and we want to put these balls into $k$ different cells.
There is a one to one correspondence between these situations and unordered factorizations
of positive integer $n$ as the form $n=n_1\times n_2\times\ldots\times n_k$
such that each factor $\geqslant1$.
In fact we can consider $n_j$ as the product of the balls in cell $j$.
There are $\alpha_j+k-1\choose k-1$ ways to put balls Labelled $p_j$.
Thus the first part is obvious.

For the second part, let $E_r$ be the set of all situations in which the cell $r$ is empty,
where $1\leqslant r\leqslant k$. Then we have
\[|E_{r_1}\cap\ldots\cap E_{r_i}|=
\prod_{j=1}^{n}{\alpha_j+k-i-1\choose k-i-1},\quad 1\leqslant i\leqslant k-1.\]
Thus the principle of inclusion and exclusion implies the result.
\end{proof}
Let $\mathcal{F}(n;k,\ell)$ denote the number of unordered factorizations
 of a positive integer $n$ into exactly $k$ factors, such that every factors $\geqslant\ell$. Means,
 the number of ways can be written positive integer $n$ as
the a product $n=n_1\times n_2\times\ldots\times  n_{k}$,
 where $n_1\geqslant n_2 \geqslant \ldots \geqslant n_{k} \geqslant\ell$.
  We call
$\mathcal{F}(n)$ is \textit{the unordered Factorization function of $n$}
(in analogy with partitions function $p(n)$ for sum).
   For example $\mathcal{F}(12)$, corresponding to
    $2\times6, 2\times2\times3, 3\times4$ and $12$.
    The sequence $\mathcal{F}(n)$ is listed in \cite{In1}.\\
    Now, by using of the using of theorems \ref{pi} and \ref{pi0}, we conclude the
    following proposition.
    \begin{proposition}
Let $n=p_1^{\alpha_1}\ldots p_r^{\alpha_r}$ be a positive integer. Then,
the number of unordered factorizations of $n$ into $k$ factors such that
each factor $\geqslant1$ and $\alpha_1+\ldots+\alpha_n\geqslant k\geqslant 1$, is given by
 \[\mathcal{F}(n,k,1)=\sum_{i=1}^{\alpha_1+\ldots+\alpha_n}\Gamma_0(\{\alpha_1,\ldots,\alpha_n\},\{i\});\]
Also, the number of unordered factorizations of $n$ into $k$ factors such that
each factor greater $1$ is given by
\[\mathcal{F}(n,k,2)=\sum_{i=1}^{\alpha_1+\ldots+\alpha_n}\Gamma(\{\alpha_1,\ldots,\alpha_n\},\{i\}).\]
    \end{proposition}

\section{Generating function of $D(n\lvert A)$}
In this section, by using of generating function we obtained the values of $D(n|A)$ for special multiset.
\begin{theorem}\label{GENN}
Let $n$ be a non-negative integer and $\mathbb{A}=\{a_1,\ldots,a_k\}$ be
 a multiset with the multiplicity mapping
$\theta$ and the background set
 $S(\mathbb{A})=\{b_1,\ldots,b_\ell\}$ which $\theta(b_i)=m_i$. The generation function of $D(n\lvert A)$ is given by
\begin{align}
\sum_{n=0}^{\infty} D(n\lvert A)x^n=\prod_{i=1}^{\ell}  \prod_{j = 1}^{m_i} \frac{x^{b_i}}{1-x^{b_i(m_i - j +1)}}
\end{align}
\end{theorem}
\begin{proof}
For each $1\leq i\leq \ell$, we want a monotonically nondecreasing sequence $n_{i,1} \leq n_{i,2} \leq \cdots \leq n_{i,m_i}$. For $2 \leq j \leq m_i$, we make the change of variables as follow $d_{i,1}= n_{i,1}$ and $d_{i,j}=n_{i,j} - n_{i,j-1}$ for $j=2,3,\ldots,m_i$.
Then the monotonically nondecreasing condition on the $(n_{i,j})_j$ becomes $d_{i,1} \geq 1$ and $d_{i,j} \geq 0$ for $1 \leq j \leq m_i$.  Observe that
\begin{align*}
\sum_{j=1}^{m_i} n_{i,j} &= (d_{i,1}) + (d_{i,1} + d_{i,2}) + \cdots + (d_{i,1} + d_{i,2} + \cdots + d_{i,m_i})  \\
&= \sum_{j=1}^{m_i} (m_i - j +1) d_{i,j}
\end{align*}

Then $D(n|A)$ is the number of ways of choosing all these $d_{i,j}$ such that
 \begin{align*}
    n &= \sum_{i=1}^{\ell} b_i \sum_{j =1}^{m_i} n_{i,j}= \sum_{i \in I} b_i \sum_{j=1}^{m_i} (m_i - j +1) d_{i,j}  \\
      &= \sum_{i=1}^{\ell} \left( b_i m_i d_{i,1} + \sum_{j=2}^{m_i} b_i (m_i - j +1) d_{i,j} \right)  \text{,}
      \end{align*}
where
$d_{i,1}\geq 1$ ($1\leq i\leq \ell$) and $d_{i,j}\geq 0$ ($1\leq i\leq \ell$, $2 \leq j \leq m_i$). So, the generating function for $D(n|A)$ is
$$  \prod_{i=1}^{\ell} \left( \frac{x^{b_i m_i}}{1-x^{b_i m_i}} \prod_{j = 2}^{m_i} \frac{1}{1-x^{b_i(m_i - j +1)}} \right),  $$
as required.
\end{proof}
By \eqref{PLP}, we have the following corollary.
\begin{corollary}
Let $n$ be positive integer and $A=\{1,1,\cdots,1\}$ be a multiset which $\theta(1)=\ell$. Then
$\sum_{n=0}^{\infty} D(n\lvert A)x^n=\frac{x^{\ell}}{(x;x)_{\ell}}$.
\end{corollary}
\begin{proof}
We can rewrite the generating function of $D(n\lvert A)$ as simpler
\begin{align*}
\prod_{i \in I} \left( \frac{x^{b_i m_i}}{1-x^{b_i m_i}} \prod_{j = 2}^{m_i} \frac{1}{1-x^{b_i(m_i - j +1)}} \right)
&= \prod_{i \in I}  \frac{x^{b_i m_i}}{1-x^{b_i m_i}} \frac{1}{\prod_{j = 2}^{m_i}1-x^{b_i(m_i - j +1)}}  \\
&= \prod_{i \in I}  \frac{x^{b_i m_i}}{\prod_{j = 1}^{m_i}1-x^{b_i(m_i - j +1)}}  \\
&= \prod_{i \in I}  \prod_{j = 1}^{m_i} \frac{x^{b_i}}{1-x^{b_i(m_i - j +1)}}  \text{.}
\end{align*}
Consider multiset $A=\{1,1,\cdots,1\}$ such that $\theta(1)=\ell$. Put $m_i=\ell$ and $b_i=1$, then
\begin{align*}
\sum_{n=0}^{\infty} D(n\lvert A)x^n=\frac{x}{1-x^{\ell}}\cdot \frac{x}{1-x^{\ell-1}}\cdot \cdots \cdot \frac{x}{1-x},
\end{align*}
as claimed.
\end{proof}
Now, we obtain another generating function for $D(n\vert A)$ by using of hypergeometric series.

Let $n$ be a non negative integer and $A=\{a_1,a_2,\ldots,a_k\}$ be a multiset. Let $1\leq n_1 \leq n_2 \leq \cdots \leq n_k$ be a positive solution of the system  $n=a_1n_1+\ldots+a_kn_k,$  such that $n_i=n_{i-1}+s_i$ where $s_i$ is non negative integers for $1\leq i \leq k$. For $\vert q\vert<1$, we can write
\begin{eqnarray*}
   \sum_{n=0}^\infty  D(n,A) q^n&=&\sum_{\begin{subarray}{c}1 \leqslant n_1 \leqslant \ldots \leqslant n_k\end{subarray}}
   q^{a_1n_1+\ldots+a_kn_k}\\
   &=&\sum_{\begin{subarray}{c}1 \leqslant n_1 \leqslant \ldots \leqslant n_k\end{subarray}}
   (q^{a_1})^{n_1}(q^{a_2})^{n_2}\ldots (q^{a_k})^{n_k}\\
   &=&\sum_{\begin{subarray}{c}1 \leqslant n_1 \leqslant \ldots \leqslant n_{k-1}\end{subarray}}
   (q^{a_1})^{n_1}(q^{a_2})^{n_2}\ldots (q^{a_{k-1}})^{n_{k-1}}(q^{a_k})^{n_{k-1}+s_k}\\
   &=&\sum_{\begin{subarray}{c}1 \leqslant n_1 \leqslant \ldots \leqslant n_{k-1}\end{subarray}}
   (q^{a_1})^{n_1}(q^{a_2})^{n_2}\ldots (q^{a_{k-1}+a_k})^{n_{k-1}}\frac{q^{a_k}}{1-q^{a_k}}\\
   &=&\sum_{\begin{subarray}{c}1 \leqslant x_1 \leqslant \ldots \leqslant x_{k-2}\end{subarray}}
   (q^{a_1})^{n_1}(q^{a_2})^{n_2}\ldots (q^{a_{k-2}+a_{k-1}+a_k})^{n_{k-2}}
   \frac{q^{a_k+a_{k-1}}}{(1-q^{a_k+a_{k-1}})(1-q^{a_k})}\\
   &=&\ldots \\
   &=&\frac{q^{\ell}}{(1-q^{a_1+a_2+\ldots+a_k})(1-q^{a_2+\ldots+a_k})\ldots
   (1-q^{a_{k-1}+a_k})(1-q^{a_k})}.
      \end{eqnarray*}

 \begin{corollary}
 Let $n$ be a non-negative integer and $A=\{1,2,2,\cdots,2\}$ be a multiset with $\theta(A)=2\ell+1$. The generation function of $D(n\lvert A)$ is given by
$\sum_{n=0}^{\infty} D(n\lvert A)x^n=\frac{x}{1-x}E_{\ell }(n)$,
where $E_{\ell}(n)$ be the number of partitions of positive integer $n$ with even parts to at most $\ell$ parts.
\end{corollary}
\begin{corollary}
Let $n$ be positive integer and $A=\{1,1,\cdots,1,2,2,\cdots,2\}$ be a multiset with $\ell$-times one and $d$ times two. Then
$\sum_{n=0}^{\infty} D(n\lvert A)x^n=p_{\ell}(n)E_d(n)$.
\end{corollary}
\begin{example}
The generating function for multisets $\{1,1,2\}$, $\{1,3,3\}$ and $\{1,2,3\}$ are
\begin{eqnarray*}
\sum_{n=0}^{\infty} D(n\lvert\{1,1,2\}))x^n &=& x^2 + x^3 + 2 x^4 + 2 x^5 + 3 x^6 + 3 x^7 + 4 x^8 + 4 x^{9} + \cdots  \text{,}  \\
    \sum_{n=0}^{\infty} D(n\lvert\{1,3,3\}))x^n &=& x^7 + x^8 + x^9 + 2x^{10}+ 2 x^{11} + 2 x^{12} + 3x^{13} + 3 x^{14} + 3 x^{15} \\
      &+& 4 x^{16} + 4 x^{17} + 4 x^{18} + 5 x^{19} + 5x^{20} + 5 x^{21} + 6 x^{22} + \cdots,  \\
    \sum_{n=0}^{\infty} D(n\lvert\{1,2,3\}))x^n &=& x^6 + x^7 + 2 x^8 + 3 x^9 + 4 x^{10} + 5 x^{11}  + 7 x^{12}+ 8 x^{13} + \cdots  \text{.}
\end{eqnarray*}
\end{example}

Let $n$ be a positive integer and $A=\{a_1,a_2,\cdots,a_k\}$ be a multiset. We denote the number of partitions of $n$ as $n=a_1n_1+a_2n_2+\ldots+a_kn_k$
which $n_i$'s are odd by $Do(n\lvert A)$.
\begin{theorem}
Let $n$ be a positive integer and $A=\{a_1,a_2,\cdots,a_k\}$ be multiset. Then
\begin{align*}
 Do(2n\lvert A)=\sum_{\substack{0\leq \theta(A) \leq n\\ \theta(A) is even}} D(2n-\theta(A) \lvert A) D(\theta(A)\lvert A),
\end{align*}
where $\theta(A)=\sum_{i=1}^k a_i$.
\end{theorem}
\begin{proof}
Let $n$ be positive integer. we have $2n=a_1n_1+a_2n_2+\ldots +a_kn_k$,
where $n_i =2r_i+1$ are odd. We can write as follow
\begin{align*}
2n&=a_1(2r_1+1)+a_2(2r_2+1)+\ldots + a_k(2r_k+1)\\
&=2r_1 a_1+2r_2 a_2+ \ldots + 2r_k a_k + a_1+a_2+\ldots + a_k.
\end{align*}
Since $2n$ is even; put $a_1+a_2+ \ldots +a_k= \theta(A)$, where $\theta(A)$ is even.
Then the number of natural partitions of $2n$ to odd parts is equal the number of natural partitions of $\theta(A)$ and the number of natural partition of $n-\theta(A)$.
\end{proof}
\section{relatively prime of $D(n\lvert A)$}
\begin{definition}
Let $n$ be positive integer and $A=\{a_1,a_2,\ldots,a_k\}$ be a multiset. We say that $D(n,A)$ is \textit{ relatively prime} if its parts form relatively prime set, that is, if we partition $n$ as $n=a_1n_1+a_2n_2+\ldots +a_kn_k$ then $(n_1,n_2,\cdots,n_k)=1$. We denote the number of such partitions of $n$ with
 $D^{r}(n,A)$.
\end{definition}
\begin{example}
We evaluate relatively prime natural number of $n=11$ with respect to multiset $\{1,1,2\}$, we have
\begin{eqnarray*}
11=1\times\mathbf{1}+1\times\mathbf{2}+2\times\mathbf{4},&&11=1\times\mathbf{1}+1\times\mathbf{4}+2\times\mathbf{3}\\
11=1\times\mathbf{1}+1\times\mathbf{6}+2\times\mathbf{2},&&11=1\times\mathbf{1}+1\times\mathbf{8}+2\times\mathbf{1}\\
11=1\times\mathbf{2}+1\times\mathbf{3}+2\times\mathbf{3},&&11=1\times\mathbf{2}+1\times\mathbf{5}+2\times\mathbf{2}\\
11=1\times\mathbf{2}+1\times\mathbf{7}+2\times\mathbf{1},&&11=1\times\mathbf{3}+1\times\mathbf{4}+2\times\mathbf{2}\\
11=1\times\mathbf{3}+1\times\mathbf{6}+2\times\mathbf{1},&&11=1\times\mathbf{1}+1\times\mathbf{6}+2\times\mathbf{2}\\
11=1\times\mathbf{4}+1\times\mathbf{5}+2\times\mathbf{1}.
\end{eqnarray*}
Then $D^r(11,\{1,1,2\})=10$.
\end{example}
\begin{lemma}\label{LL}
Let $n$ be a positive integer and $A=\{a_1,a_2\}$. If $a_1=a_2=a$ then $D_0(n,\{a_1,a_2\})=\lfloor\frac{n}{2a}\rfloor+1$, and if $a_1\neq a_2$ then $D_0(n,\{a_1,a_2\})=\lfloor\frac{n+a_1+a_2-1}{a_1a_2}\rfloor$.
\end{lemma}
\begin{theorem}\label{MU}
Let $n$ be a non-negative integer. For multiset $A=\{a_1,a_2,\ldots,a_k\}$ we have
\begin{align}
D^r(n,A)=\sum_{ d\lvert n}\mu(d)D\left(\frac{n}{d},A\right),
\end{align}
where $\mu(d)$ is the M\"{o}bius function.
\end{theorem}
\begin{proof}
For non-negative integers $n,k$, we have
$D(n,A)=\sum_{ d\lvert n}D^r\left(\frac{n}{d},A\right)$,
which by the \textit{M\"{o}bius inversion formula} we have that
$D^r(n,A)=\sum_{ d\lvert n}\mu(d)D\left(\frac{n}{d},A\right)$,
as required.
\end{proof}

\begin{corollary}
Let $n$ be non-negative integer and $A=\{a_1,a_2\}$.  If $a_1=a_2=a$ then
\[ D_0^r(n,\{a_1,a_2\})=\frac{1}{2a}\lfloor \varphi(n)\rfloor,\]
where $\varphi(n)$ is the \textit{Eulers totient function}.
\end{corollary}
\begin{proof}
Let $n, k$  be non-negative integers and $p_1^{\alpha_1}\ldots p_k^{\alpha_k}$
be the prime decomposition of $n$. By Lemma~\ref{LL} and Theorem~\ref{MU}, we have
$D^r(n,A)=\sum_{ d\lvert n}\mu(d)\big{(}\lfloor \frac{n}{2ad}\rfloor+1 \big{)}$.
If $2ad\lvert n$ then $\lfloor\frac{n}{2ad}\rfloor$ be integer and recall that $ \sum_{ d\lvert n} \varphi(n)=n$ and $\sum_{ d\lvert n}\mu(d)=\lfloor \frac{1}{n} \rfloor$, by the M\"{o}bius inversion we have
\[\sum_{ d\lvert n}\mu(d)\lfloor\frac{n}{2ad}\rfloor=\frac{1}{2a} \varphi(n).\]
Now, if $2ad \nmid n$ we have
\begin{align*}
\sum_{ d\lvert n}\mu(d)\lfloor \frac{n}{2ad}\rfloor&=\sum_{ d\lvert n}\mu(d)(\frac{n}{2ad}-\frac{1}{2a})=\sum_{ d\lvert n}\mu(d)(\frac{n}{2ad})-\sum_{ d\lvert n} \mu(d)(\frac{1}{2a})\\
&=\frac{1}{2a} \varphi(n) -\frac{1}{2a}\sum_{ d\lvert n}\mu(d)=\frac{1}{2a} \varphi(n),
\end{align*}
as claimed.
\end{proof}

{\bf Acknowledgement}: The third author would like to thanks Eric Towers for discussing some of the results of this paper.

\end{document}